\documentclass[11pt]{article}
\usepackage{amsthm}
\usepackage{amsmath}
\usepackage{amssymb}

\newtheorem{theorem}{Theorem}[section]
\newtheorem{lemma}[theorem]{Lemma}
\newtheorem{corollary}[theorem]{Corollary}
\newtheorem{proposition}[theorem]{Proposition}

\DeclareMathOperator{\conv}{conv}

\DeclareMathOperator{\cl}{cl}

\newcommand{\R}{\mathbb{R}}
\newcommand{\BR}{ {\overline{\mathbb{R}} }  }

\newcommand{\inner}[2]{\langle{#1},{#2}\rangle}

\newcommand{\norm}[1]{\|#1\|}

\newcommand{\tos}{\rightrightarrows} 

\begin{document}

\title{Geometric properties of maximal monotone operators and  convex
functions which may represent them}

\author{B. F. Svaiter\thanks{IMPA, Estrada Dona Castorina 110, 
   22460-320 Rio de Janeiro, Brazil
 (email: benar@impa.br). 
This work was Partially supported by CNPq grants no.\
303583/2008-8,
302962/2011-5,
480101/2008-6,
474944/2010-7,
and FAPERJ grants E-26/102.821/2008, E-26/102.940/2011.
}}

\maketitle

\begin{abstract}
  We study the relations between some geometric properties
  of maximal monotone operators and generic geometric and analytical
  properties of the functions on the associate Fitzpatrick family of
  convex representations.
  We also investigate under which conditions a convex function represents a
  maximal monotone operator with bounded range
  and provide an example of a non type (D) operator on this class. 
\end{abstract}

\maketitle

\section{Introduction}
\label{sec:int}

Fitzpatrick functions~\cite{MR1009594} are convex functions which
represent maximal monotone operators.  A natural question is whether
geometric features of maximal monotone operators are related with
generic analytical or geometric features of their convex
representations.  We present in this work some answers to this
question.  Using these results, we prove that the convex hull of the
range of bounded domain maximal monotone operators are weak-$*$ dense.

Another natural question is whether a convex function represents a
maximal monotone operator.  Although this question is reasonably
settled in reflexive Banach spaces~\cite{MR1974634}, in non-reflexive
Banach spaces we have, up to now, what seems to be partial
results~\cite{MR2489609,MR2583901,MR2731293,MR2675662} ( see
also~\cite{MR2559952,MR2777606}). In this work we answer this question
for the case of bounded-range maximal monotone operators and show
that there exist non-type (D) operators in this class.

This work is heavily based on results and technique previously developed
in collaboration with J. E. Mart\'\i
nez-Legaz~\cite{MR2128696,MR2361859}, M. Marques
Alves~\cite{MR2489609,MR2583901,MR2731293,MR2675662,MR2559952,MR2777606}
and O. Bueno~\cite{nodliniso}.

This work is organized as follows. In Section~\ref{sec:nbr} we
establish the notation and recall some basic results.  In
Section~\ref{sec:dom} we analyze some properties of the domains of
extended real valued functions defined in the Cartesian product of a
Banach space with its dual.  In Section~\ref{sec:gp} we characterize
the properties of the Fitzpatrick families associated to bounded-range
and bounded-domain maximal monotone operators.  In
Section~\ref{sec:main} we characterize the convex functions which
represent bounded range maximal monotone operators and provide an
example of a non-type (D) operator in this class.

Throughout this work, $\BR$ stands for the extended real set
$\R\cup\{\infty,-\infty\}$ and, 
in a Cartesian product $A\times B$ we will use the notation $P_1$ and
$P_2$ for the canonical projections onto the first and onto the second
term of the product respectively, that is
\[
P_1:A\times B\to A, \;P_1(a,b)=a,\quad
P_2:A\times B\to B, \;P_2(a,b)=b.
\]

\section{Notation and Basic Results}
\label{sec:nbr}

Let $X$ be a real Banach space, possible non-reflexive, with
topological dual denoted by $X^*$ and endowed with the canonical
dual norm. 
We will use the notation $\inner{\cdot}{\cdot}$ for the duality
product in $X\times X^*$,
\[
  \inner{x}{x^*}=\inner{x^*}{x}=x^*(x),\qquad
  x\in X,x^*\in X^*.
\]
In this work, 
 $X^{**}$, $X^{***}$, etc., stands for
$(X^*)^*$, $((X^*)^*)^*$, etc.
Recall that $X^*$, $X^{**}$,\dots,\, are Banach spaces.
We will identify $X$ with its image by the canonical injection of $X$
into $X^{**}$, that is, each $x\in X$ is identified with the
(continuous, linear) functional
\[
X^*\to\R,\quad x^*\mapsto\inner{x}{x^*}.
\]
The same identification will be used for $X^*$, $X^{**}$, etc. with respect to 
$X^{***}$, $X^{****}$, etc.
The Banach space $X$ is \emph{reflexive} if this
injection, of $X$ into $X^{**}$, is onto, and non-reflexive otherwise.
The weak-$*$ topology of $X^*$ is the smallest topology (in $X^*$) in
which the functional $x^*\mapsto \inner{x}{x^*}$ are continuous for
all $x\in X$.

A closed ball in $X$, with radius $r$ and center $0$,
will be denoted by $B_X[r]$,
\[
B_X[r]=\{x\in X\;|\; \norm{x}=r\}.
\]
The closure and convex hull of $C\subseteq X$ will be
denoted by $\cl\, C$ and $\conv C$, respectively.  The closure of
$C\subset X^*$ in the weak-$*$ topology will be denoted by
$\cl_{w*}(C)$. 
The set of 
\emph{directions of recession} of $C\subseteq X$
or \emph{recession cone} of $C$ is defined as
\[
0^+C=\{u\in X\;|\; x+\lambda u\in C,\quad\forall x\in C\}.
\]
The \emph{indicator function} of $C\subseteq X$ is
\begin{equation}
  \label{eq:if}
\delta_C:X\to\BR,\qquad\delta_{C}(x)=
\begin{cases}
  0,& x\in C\\
  \infty,&\text{otherwise}.
\end{cases}
\end{equation}

A point-to-set operator $T:X\tos X^*$ is a relation $T\subset X\times
X^*$ and
\[
 T(x)=\{x^*\in X^*\;|\; (x,x^*)\in T\},\qquad x\in X.
\]
The domain and range of $T:X\tos X^*$ are (respectively)
\[
D(T)=\{x\in X\;|\; T(x)\neq\emptyset\}=P_1T,\qquad
R(T)=\{x^*\in X\;|\; \exists x\in X,\, x^*\in T(x)\}=P_2T.
\]
An operator $T:X\tos X^*$ is \emph{monotone} if
\[
\inner{x-y}{x^*-y^*}\geq 0,\qquad \forall (x,x^*),(y,y^*)\in T
\]
and it is \emph{maximal monotone} if it is monotone and maximal in the
family of monotone operators with respect to the inclusion.  A maximal
monotone $T:X\tos X^*$ is of \emph{Gossez type (D)}~\cite{MR0313890}
if any point in the set
\[
\{(x^{**},x^{*})\in X^{**}\times X^*\;|\;
\inner{x^{**}-y}{x^*-y^*}\geq 0,\quad\forall (y,y^*)\in T\}
\]
is the weak-$*\times$strong limit of a \emph{bounded} net of points in $T$.

Let $f:X\to\BR$. The function $f$ is \emph{proper} if
$f\not\equiv\infty$ and $f>-\infty$.  The function $f$ is
\emph{closed} if it is lower semi-continuous and $\cl f$, the closure
of $f$, is the largest closed function majorized by $f$.
The \emph{domain} of $f$ is
\[
D(f)=\{x\in X\;|\; f(x)<\infty\}.
\]
We use the notation $\conv f$ for the lower convex envelop of $f$,
which is the largest convex function majorized by $f$. The \emph{conjugate}
of $f:X\to\BR$ is
\[
f^*:X^*\to\BR,\quad f^*(x^*)=\sup_{x\in X}\inner{x}{x^*}-f(x).
\]
Latter on we will need the following elementary result
\begin{proposition}
  \label{pr:e}
  Let $f:X\to\BR$ be a proper closed convex function. For any $\lambda\geq 0$,
  $z^*,w^*\in X^*$ 
  \[
  f^*(z^*)+\lambda\delta_{D(f)}^*(w^*)\geq f^*(z^*+\lambda w^*)
  \]
  where $\delta_{D(f)}^*$ stands for $(\delta_{D(f)})^*$, the support function of
  $D(f)$.
\end{proposition}

\begin{proof}
  See Appendix.
\end{proof}

In the study of convex functions in $X\times X^*$ it is convenient to
define the operator $\mathcal{J}$~\cite{MR1934748},
\begin{equation}
  \label{eq:j}
  \mathcal{J}:\BR^{X\times X^*}\to\BR^{X\times X^*},\quad
  (\mathcal{J}\,h)(x,x^*)=h^*(x^*,x).
\end{equation}
As remarked in~\cite{MR1934748}, 
the operator $\mathcal{J}$ is
a generalized Moreau conjugation by means of the (symmetric) coupling
function 
\[
\Phi:(X\times X^*)\times (X\times X^*)\to\R,\qquad
\Phi((x,x^*),(y,y^*))=\inner{x}{y^*}+\inner{y}{x^*}
\]
so that
\[
\mathcal{J}h(z)=h^\Phi(z)=\sup_{z'\in X\times X^*}
 \Phi(z,z')-h(z').
\]
Hence, we may call the operator $\mathcal{J}$ also the $\Phi$-conjugation.
It will also be useful to have a
notation for the (ordered) duality product, say
\begin{equation}
  \label{eq:pi}
\pi:X\times X^*\to\R,\quad \pi(x,x^*)=\inner{x}{x^*},
\end{equation}

Let $T:X\tos X^*$ be a maximal monotone operator.  Fitzpatrick family
$\mathcal{F}_T$ associated with $T$ is defined as
\begin{equation}
  \label{eq:fitz.fm}
  \mathcal{F}_T=\left\{h:X\times X^*\to\BR\;\left|
      \begin{array}{l}
        h\text{ is convex and closed}\\
       h(x,x^*)\geq \inner{x}{x^*},\quad\forall (x,x^*)\in X\times X^*\\
      (x,x^*)\in T\Rightarrow h(x,x^*)=\inner{x}{x^*}
      \end{array}
   \right\}\right. .
\end{equation}
The smallest function in $\mathcal{F}_T$ (which is always non-empty)
is Fitzpatrick minimal function $\varphi_T$
\begin{equation}
  \label{eq:fitz.fc}
  \varphi_T:X\times X^*\to\BR,\qquad
  \varphi_T(x,x^*)=\sup_{(y,y^*)\in T}\inner{x}{y^*}+\inner{y}{x^*}
  -\inner{y}{y^*}. 
\end{equation}
Any function in the family $\mathcal{F}_T$ characterizes $T$ in the
following sense.

\begin{theorem}[\mbox{\cite[Theorem 3.10]{MR1009594}}]
  \label{th:fitz}
  Let $T:X\tos X^*$ be maximal monotone. For any $h\in\mathcal{F}_T$,
  \begin{equation*}
    T=\{(x,x^*)\in X\times X^*\;|\; h(x,x^*)=\inner{x}{x^*}\}.
  \end{equation*}
\end{theorem}
\noindent

For any maximal monotone $T:X\tos X^*$, the largest element of
$\mathcal{F}_T$ is ${\sigma}_T$~\cite{MR1934748}
\begin{equation}
  \label{eq:fitz.max}
  {\sigma}_T:X\times X\to\BR,\quad{\sigma}_T=\cl\,\conv\;(\pi+\delta_T)
\end{equation}
and
\begin{equation}
  \label{eq:usefull}
  \varphi_T (x,x^*)={{\sigma}_T}^*(x^*,x)=(\pi+\delta_T)^*(x^*,x),
  \qquad\forall (x,x^*)\in X\times X^*.
\end{equation}
The elementary properties of ${\sigma}_T$, together with its inclusion
in $\mathcal{F}_T$ will be instrumental for proving Lemma~\ref{lm:dtcl}.
Using \eqref{eq:j} we obtain a very simple reformulation of
the above relation,
\begin{equation}
  \label{eq:usefull.2}
  \varphi_T=\mathcal{J}{\sigma}_T=\mathcal{J}(\pi+\delta_T).
\end{equation}
The family $\mathcal{F}_T$ is invariant
under $\mathcal{J}$, that is $\mathcal{J}\left(\mathcal{F}_T\right)\subseteq
\mathcal{F}_T$. In particular, for any $h\in\mathcal{F}_T$,
\begin{equation}
  \label{eq:icb}
  h\geq \pi,\qquad\mathcal{J}h\geq \pi.  
\end{equation}
In \emph{reflexive} Banach spaces the two above inequalities are
a sufficient condition for a proper closed convex function $h$ to
represent a maximal monotone operator~\cite{MR1934748}. In
non-reflexive Banach spaces, up to now, additional conditions are
required on $h$~\cite{MR2489609,MR2583901,MR2731293,MR2675662}.  The
next result was essentially proved in~\cite[Theorem 3.1, Corollary
3.2]{MR2675662}. We will present it as stated in~\cite[Theorem
2.1]{QSUM}

\begin{theorem}
  \label{th:hhs}
  If $h:X\times X^*\to\BR$ is a closed convex function,
 $h\geq\pi$, $\mathcal{J}h\geq\pi$
 and, for some $x_0\in X$,
 \[
 \bigcup_{\lambda>0}\lambda (P_1D(h)-x_0)
 \]
 is a closed subspace, then
 \[
  T=\{\zeta\in X\times X^*\;|\; \mathcal{J}h(\zeta)=\pi(\zeta)\}
 \]
 is maximal monotone, $\mathcal{J}h\in\mathcal{F}_T$. If, additionally
 $h$ is lower semi-continuous in the strong$\times$weak-$*$ topology, then
 $h\in\mathcal{F}_T$.
\end{theorem}

\section{Extended real valued
functions  in $X\times X\sp*$ and their conjugates}
\label{sec:dom}

In this section we study extended real valued functions in $X\times
X^*$.
We start with a general and simple result.

\begin{proposition}
  \label{pr:dr}
  If $h:X\times X^*\to\BR$ is a proper closed convex function then $
  D(\mathcal{J}\delta_{D(h)})\subseteq 0^+D(\mathcal{J}h) $.
\end{proposition}

\begin{proof}
  Use definition~\eqref{eq:j} and Proposition~\ref{pr:e} with $f=h$ 
\end{proof}

Now we study those extended real valued functions in
$X\times X^*$ which satisfy one of the inequalities in~\eqref{eq:icb}.
The two next propositions are the main technical tools in this work.

\begin{lemma}
  \label{lm:dhs}
  If $h:X\times X^*\to\BR$ and $\mathcal{J}h\geq\pi$
  (where 
  $\mathcal{J}$, $\pi$ are as in \eqref{eq:j}, \eqref{eq:pi}), then
 \begin{align*}
  P_2(D\left(\mathcal{J}\,h\right))&=
  P_1\left( D (h^*) \cap (X^*\times X\right))
  \subseteq\cl_{w*} \conv P_2(D(h)),\\
  P_1(D\left(\mathcal{J}\,h\right))&=
  P_2\left(D(h^*)\cap (X^*\times X\right))
  \subseteq\cl \conv P_1(D(h)).
 \end{align*}
\end{lemma}

\begin{proof}
  If $h(x,x^*)=-\infty$ for some $(x,x^*)$ the proposition holds
  trivialy.  Assume that $h$ is proper. The two equalities follows
  triviality from definition~\eqref{eq:j}.  To prove the first
  inclusion, suppose that
  \begin{equation}
    \label{eq:abs.1}
    u^*\notin\cl_{w*}\conv P_2(D(h)),
    \quad 
    u^*\in P_1(D(h^*)\cap (X^*\times X)).
  \end{equation}
  From the first above relation and the geometric form of Hahn-Banach
  Lemma in $X^*$ endowed with the weak-$*$ topology it follows that
  there exist $\hat x\in X$ and $\varepsilon>0$ such that
  \[
  \inner{\hat x}{u^*}\geq
   \inner{\hat x}{y^*}+\varepsilon,
  \quad \forall y^*\in P_2(D(h^*)),
  \] 
  while the second relation in \eqref{eq:abs.1}  means that
  $\infty>h^*(u^*,u)\geq \inner{u^*}{u}$ for some $u\in X$. By the
  definition of the conjugation
  \[
   h^*(u^*,u)\geq\inner{y}{u^*}+\inner{u}{y^*}-h(y,y^*),
   \qquad   \forall (y,y^*)\in X\times X^*.
  \]
  Combining the two above inequalities we conclude that, for any
  $\lambda\geq 0$, 
  \[
  h^*(u^*,u)+\lambda[\inner{\hat x}{u^*}-\varepsilon]
  \geq \inner{y}{u^*}+\inner{u+\lambda \hat x}{y^*}-h(y,y^*).
  \]
  Therefore, taking the $\sup$ for $(y,y^*)\in X\times X^*$ in the right hand
  side of the above inequality we have
  \[
  h^*(u^*,u)+\lambda[\inner{\hat x}{u^*}-\varepsilon] 
  \geq h^*(u^*,u+\lambda\hat x)
  \geq \inner{u+\lambda\hat x}{u^*}.
  \]
  Dividing these  inequalities by $\lambda$,
  noting that $h^*(u^*,u)\in\R$, and taking the limit
  $\lambda\to\infty$ we obtain $\inner{\hat x}{u^*}-\varepsilon
  \geq \inner{\hat x}{u^*}
  $
  which is absurd. Therefore, there is no $u^*$ as in \eqref{eq:abs.1}
  and the first inclusion on the lemma holds.

  To prove the second inclusion, suppose that
  \begin{equation}
    \label{eq:abs.2}
    u\notin \cl\,\conv\,P_1(D(h)),
    \qquad
    u\in P_2(D(h^*)\cap (X^*\times X)).
  \end{equation}
  From the firs above relation and the geometric form of Hahn-Banach
  Lemma in $X$ endowed with the strong topology it follows that 
  there exist $\hat x^*\in X^*$ and $\varepsilon>0$ such that
  \[
  \inner{u}{\hat x^*}\geq\inner{y}{\hat x^*}+\varepsilon,\qquad \forall
  y\in P_1(D(h)),
  \]
  while the second relation in \eqref{eq:abs.2} means that
  there exists $u^*\in X^*$ such that $\infty>h^*(u^*,u)\geq\inner{u}{u^*}$.
  Combining this inequality with the definition of the conjugation
   we conclude that for any $\lambda\geq 0$
  \[
  h^*(u^*,u)+\lambda[\inner{u}{\hat x^*}-\varepsilon]\geq
  \inner{y}{u^*+\lambda \hat x^*}+\inner{u}{y^*}-h(y,y^*).
  \]
  Therefore, taking the $\sup$ in $y\in X, y^*\in X^*$ at the right hand
  side of the above inequality we have
  \[
   h^*(u^*,u)+\lambda[\inner{u}{\hat x^*}-\varepsilon]\geq
   h^*(u^*+\lambda\hat x^*,u)\geq \inner{u}{u^*+\lambda\hat x^*}.
   \]
  Dividing this inequality by $\lambda$,
  noting that $h^*(u^*,u)\in\R$, and taking the limit
  $\lambda\to\infty$ we obtain $\inner{u}{\hat x^*}-\varepsilon
  \geq \inner{u}{\hat x^*}
  $
  which is absurd. Therefore, there is no $u^*$ as in \eqref{eq:abs.2}
  and the second inclusion on the lemma holds.
\end{proof}

The next result is proved using the same reasoning as in Lemma~\ref{lm:dhs}.

\begin{lemma}
  \label{lm:dh}
  If $h:X\times X^*\to\BR$ is convex and $h\geq\pi$ then
  \[
  P_2(D(h))\subseteq\cl_{w*} P_1(D(h^*)), \qquad
  P_1(D(h))\subseteq\cl_{w*} P_2(D(h^*)).
  \]
\end{lemma}

\begin{proof}
  If $h\equiv\infty$ the lemma holds trivially.
  Since $\cl h$ is convex, majorizes the duality product, 
  $h^*=(\cl\, h)^*$ and $D(h)\subset D(\cl\, h)$,
  we may assume that $h$ is a proper closed convex function and,
  under this assumption, using Moreau Theorem we have
  \[
   h(x,x^*)=h^{**}(x,x^*).
  \]

  Suppose that
  \[
  u^*\notin\cl_{w*}\conv P_1(D(h^*)),
  \quad u^*\in
  P_2(D(h)).
  \]
  This means that there exists $\hat x\in X$ and $\varepsilon>0$
  such that
  \[
  \inner{\hat x}{u^*}\geq
   \inner{\hat x}{y^*}+\varepsilon,
  \quad \forall y^*\in P_1(D(h^*)),
  \] 
  and that $\infty>h(u,u^*)\geq \inner{u}{u^*}$ for some
  $u\in X$. Hence, by the  definition of the conjugation 
  \[
   h(u,u^*)=h^{**}(u,u^*)\geq\inner{y^{**}}{u^*}+\inner{u}{y^*}-h^*(y^*,y^{**})
   \qquad   \forall y^*\in X^*, y^{**}\in X^{**}.
  \]
  Combining the two above equations we conclude that, for any
  $\lambda\geq 0$, 
  \[
  h(u,u^*)+\lambda[\inner{\hat x}{u^*}-\varepsilon]
  \geq \inner{y^{**}}{u^*}+\inner{u+\lambda \hat x}{y^*}-h^*(y^*,y^{**}).
  \]
  Therefore, taking the $\sup$ in $y^*\in X^*, y^{**}\in X^{**}$ at
  the right hand side of the above inequality we have
  \[
  h(u,u^*)+\lambda[\inner{\hat x}{u^*}-\varepsilon] 
  \geq h^{**}(u+\lambda\hat x,u^*)=h(u+\lambda\hat x,u^*)
  \geq \inner{u+\lambda\hat x}{u^*}.
  \]
  Dividing this inequality by $\lambda$ and taking the limit
  $\lambda\to\infty$ we obtain $\inner{\hat x}{u^*}-\varepsilon
  \geq \inner{\hat x}{u^*}
  $
  which is absurd. 

  Suppose that
  \[
  u\notin\cl_{w*}\conv P_2(D(h^*)),\quad u\in
  P_1(D(h)).
  \]
  This means that there exists $\hat x^*\in X^*$ and $\varepsilon>0$
  such that
  \[
  \inner{u}{\hat x^*}\geq
   \inner{y^{**}}{\hat x^*}+\varepsilon,
  \quad \forall y^{**}\in P_2(D(h^*)),
  \] 
  and that $\infty>h(u,u^*)\geq \inner{u}{u^*}$ for some
  $u^*\in X^*$. Hence, by the  definition of the conjugation 
  \[
   h(u,u^*)=h^{**}(u,u^*)\geq\inner{y^{**}}{u^*}+\inner{u}{y^*}-h^*(y^*,y^{**})
   \qquad   \forall y^*\in X^*, y^{**}\in X^{**}.
  \]
  Combining the two above equations we conclude that, for any
  $\lambda\geq 0$, 
  \[
  h(u,u^*)+\lambda[\inner{u}{\hat x^*}-\varepsilon] \geq
  \inner{y^{**}}{u^*+\lambda \hat x^*}+\inner{u}{y^*}-h^*(y^*,y^{**}).
  \]
  Therefore, taking the $\sup$ in $y^*\in X^*, y^{**}\in X^{**}$ at
  the right hand side of the above inequality we have
  \[
  h(u,u^*)+\lambda[\inner{u}{\hat x^*}-\varepsilon]
  \geq h^{**}(u,u^*+\lambda\hat x^*)=h(u,u^*+\lambda\hat x^*)
  \geq \inner{u}{u^*+\lambda\hat x^*}.
  \]
  Dividing this inequality by $\lambda$ and taking the limit
  $\lambda\to\infty$ we obtain $\inner{u}{\hat x^*}-\varepsilon \geq
  \inner{u}{\hat x^*}$, which is absurd.
\end{proof}  

Now we analyze those functions $h$ for which $P_1D(h)$ or $P_2D(h)$
(or $P_1D(h^*)$, $P_2D(h^*)$) is bounded.

\begin{proposition}
  \label{pr:hb}
  Let  $h:X\times X^*\to\BR$  and $0\leq L<\infty$.
  \begin{enumerate}
  \item 
    If   $P_2D(h)\subseteq B_{X^*}[L]$ then
    \[
    h^*(x^*,x^{**})\leq h^*(x^*,z^{**})+L\norm{x^{**}-z^{**}},\qquad
    \forall x^*\in X^*,x^{**},z^{**}\in X^{**}
    \]
    and  $D(h^*)=(P_1D(h^*))\times X^{**}$;
  \item    
    If $P_1D(h)\subseteq B_{X}[L]$ then
    \[
    h^*(x^*,x^{**})\leq h^*(z^*,x^{**})+L\norm{x^*-z^*},\qquad
    \forall x^*,z^*\in X^*,x^{**}\in X^{**}
    \]
    and  $D(h^*)=X^*\times (P_2D(h^*))$.
  \end{enumerate}
\end{proposition}

\begin{proof}
  Using the definition of the conjugate and the assumption on the domain
  of $h$ we conclude that for any $x^*\in X^{**}$ and
  $x^{**},z^{**}\in X^{**}$
  \begin{align*}
    h^*(x^*,x^{**})
    &=\sup_{(y,y^*)\in X\times X^*}
     \inner{y}{x^*}+\inner{x^{**}}{y^*}-h(y,y^*)\\
    &=\sup_{(y,y^*)\in X\times X^*,\,\norm{y^*}\leq L}
     \inner{y}{x^*}+\inner{x^{**}}{y^*}-h(y,y^*)\\
    &=\sup_{(y,y^*)\in X\times X^*, \norm{y^*}\leq L}
     \inner{y}{x^*}+\inner{z^{**}}{y^*}-h(y,y^*)+\inner{x^{**}-z^{**}}{y^*}\\
    &\leq\sup_{(y,y^*)\in X\times X^*, \norm{y^*}\leq L}
     \inner{y}{x^*}+\inner{z^{**}}{y^*}-h(y,y^*)+L\norm{x^{**}-z^{**}}
    =h^*(x^*,z^{**})+L\norm{x^{**}-z^{**}}
  \end{align*}
  which proves the first result in item 1). The second result in
  item 1 follows triviality from the first result.

  Item 2 is proved by the same reasoning.
\end{proof}

\begin{proposition}
  \label{pr:hsb}
  Suppose that $h:X\times X^*\to\BR$ is a proper closed convex function
  and $0\leq L\leq\infty$.
  \begin{enumerate}
  \item   If
  $P_1D(h^*)\subseteq B_{X^*}[L]$ then
  for any $x^*\in X^*$ and $x,z\in X$
  \[
  h(x,x^*)\leq h(z,x^*) +L\norm{x-z},\qquad\forall x^*\in X^*, x,z\in X
  \]
  and $D(h)=X\times (P_2D(h))$;
  \item   if
  $P_2D(h^*)\subseteq B_{X^{**}}[L]$ then
  for any $x^*\in X^*$ and $x,z\in X$
  \[
  h(x,x^*)\leq h(x,z^*) +L\norm{x^*-z^*},\qquad\forall x^*,z^*\in X^*, x\in X
  \]
  and $D(h)=(P_1D(h))\times X^*$.
  \end{enumerate}
\end{proposition}

\begin{proof}
  Since $h$ is a proper closed convex function,
  according to Moreau Theorem we have
  \[
  h(x,x^*)=h^{**}(x,x^*),\qquad \forall (x,x^*)\in X\times X^*.
  \]
  Items 1 and 2 follows directly from Proposition~\ref{pr:hb}
  applied to $h^*$ and the above equation.
\end{proof}

We end this section combining the previous result.  

\begin{corollary}
  \label{cr:bas}
  If $h:X\times X^*\to\BR$, 
 $\mathcal{J}h\geq \pi$ 
 and 
$P_2(D(h))\subseteq B_{X^*}[L]$
  then\\
  1) $h^*(x^*,\cdot)$ is $L$-Lipschitz continuous for any
  $x^*\in P_1(D(h^*))$;\\
  2) $D(h^*)=P_1(D(h^*))\times X^{**}\subseteq B_{X^*}[L]\times X^{**}$;\\
  3) $P_1D(h^*)\subset B_{X^*}[L]$, $P_2D(h^*)=X^{**}$.
\end{corollary}

\begin{proof}
  Item 1), and the equality in item 2) follows from
  Proposition~\ref{pr:hb}(item 1). In particular
  \[
  D(h^*)\cap(X^*\times X)=(P_1D(h^*))\times X.
  \]
  To end the prove of item 2), use the above equality and the first
  result in Lemma~\ref{lm:dhs}. Item 3) follows trivially from item
  2).
\end{proof}

\begin{corollary}
  \label{cr:bas2}
  If $h:X\times X^*\to\BR$ is a proper closed convex
  function, $h\geq\pi$ and $P_1(D(h^*))\subseteq B_{X^*}[L]$
  then\\
  1) $h(\cdot,x^*)$ is $L$-Lipschitz continuous for any
  $x^*\in P_2(D(h))$;\\
  2) $D(h)=X\times P_2(D(h))\subseteq X\times B_{X^*}[L]$;\\
  3) $P_2D(h)\subset B_{X^*}[L]$, $P_1(Dh)=X$.
\end{corollary}

\begin{proof}
  Item 1), and the equality in item 2) follows from
  Proposition~\ref{pr:hsb}(item 1), while the inclusion
  in item 2) follows from  Lemma~\ref{lm:dh}.
  Item 3) follows trivially from item 2).
\end{proof}

\section{Some geometric properties of maximal monotone operators
and their  relations with their Fitzpatrick families}
\label{sec:gp}

In this section we discuss the relation of some elementary
geometric properties of maximal monotone operators  with
elementary geometric and analytical (continuity) properties of the
functions on Fitzpatrick families of these operators.

First we prove invariance of the closure of the projections of the domains
along each Fitzpatrick family.
\begin{lemma}[Invariant features of the domains]
  \label{lm:dtcl}
  Let $T:X\tos X^*$ be maximal monotone. For any $h\in
  \mathcal{F}_T$,
  \[
   \cl\;P_1D(h)=\cl\;\conv\;D(T),\qquad
   \cl_{w*} P_2D(h)=\cl_{w*}\conv\;R(T).
  \]
\end{lemma}

\begin{proof}
  Applying Lemma~\ref{lm:dhs} to ${\sigma}_T$,
  $\varphi_T=\mathcal{J}{\sigma}_T$ and using
  definition~\eqref{eq:fitz.max} we conclude that
  \[
   \cl\,P_1D(\varphi_T)\subseteq\cl\, P_1(D{\sigma}_T)=
   \cl\;\conv\;D(T),\qquad
   \cl_{w*} P_2(\varphi_T)\subseteq
   \cl_{w*} P_2({\sigma}_T)=
\cl_{w*}\conv\;R(T).
  \]
  Take $h\in\mathcal{F}_T$.
  Since $\varphi_T\leq h$,
  \[
  T\subseteq D(h)\subseteq D(\varphi_T)
  \]
   where the first inclusion follows from Theorem~\ref{th:fitz}.
   To end the proof, combine the two above equations.
\end{proof}

Lemma~\ref{lm:dtcl} is tight in the following sense: for a maximal
monotone operator $T$ and $h\in\mathcal{F}_T$,
\begin{equation}
  \label{eq:tg}
T\subseteq D(h)\subseteq \cl\;\conv P_1D(T)\times 
\cl_{w^*}\conv P_2
\end{equation}
and these inclusion may hold as equalities. For example,
let
\[
T:\R\to\R,\quad T(x)=x.
\]
Then $T\subsetneq D(T)\times R(T)=\R^2$, and 
 the first and the second  inclusion in
\eqref{eq:tg} holds as equalities for $h={\sigma}_T$ and
$h=\varphi_T$ respectively. Indeed, for this operator,
\[
{\sigma}_T(x,x^*)=x^2+\delta_{0}(x-x^*),\qquad
\varphi_T(x,x^*)=(x+x^*)^2/4.
\]

It had long been known that there exists a duality relation
between the support function of the domain and the recession
directions of the conjugate~\cite{MR0274683}.
Lemma~\ref{lm:dtcl} will be used to prove that
there exists a kind of ``duality relation'' between the domain
and the range of a maximal monotone operator.

\begin{lemma}[Domain/Range duality relation]
  \label{pr:rds}
  If $T:X\tos X^*$ is maximal monotone then
  \[
    D\left( \delta_{D(T)}^*\right)\subseteq 0^+ ( \cl_{w*}\;\conv R(T)  ),\qquad
     D\left(\delta_{R(T)}^*\right)\cap X\subseteq 0^+ ( \cl\;\conv D(T)  ).
  \]
\end{lemma}

\begin{proof}
  To prove the first inclusion, suppose that $u^*\in
  D\left(\delta_{D(T)}^*\right)$.  Then
  \begin{align*}
    \infty>\delta_{D(T)}^*(u^*)=\delta_T^*(u^*,0)
    =(\mathcal{J}\delta_T)(0,u^*)=(\mathcal{J}
    \delta_{D(\sigma_T)})(u^*,0),
  \end{align*}
  and it follows from Proposition~\ref{pr:dr} that $(0,u^*)\in
  0^+D(\mathcal{J}\sigma_T)=0^+D(\varphi_T)$. Therefore,
  $u^*\in 0^+P_2D(\varphi_T)$ and the conclusion follows from
  this inclusion and Lemma~\ref{lm:dtcl}.

  To prove the second inclusion, suppose that $u\in D(\delta_{R(T)}^*)\cap X$.
  Then
  \begin{align*}
    \infty>\delta_{R(T)}^*(u)=\delta_T^*(0,u)
    =(\mathcal{J}\delta_T)(u,0)=\left(\mathcal{J}
    \delta_{D(\sigma_T)}\right)(u,0),
  \end{align*}
  and it follows from Proposition~\ref{pr:dr} that $(u,0)\in
  0^+D(\mathcal{J}\sigma_T)=0^+D(\varphi_T)$. Therefore,
  $u\in 0^+P_1D(\varphi_T)$ and the conclusion follows from
  this inclusion and Lemma~\ref{lm:dtcl}.
\end{proof}

In view of the domain/range duality expressed in Lemma~\ref{pr:rds},
maximal monotone operators with bounded domains/ranges shall have
unbounded range/domains. One of these results is already well know.
The other (which we do not know if it is new) is proved next.

\begin{corollary}
  \label{cr:g}
  If $T:X\tos X^*$ is maximal monotone and $D(T)$ is bounded, then
  $\conv R(T)$ is weak-$*$ dense in $X^*$.
\end{corollary}

\begin{proof}
  Note that if $D(T)$ is bounded then $D((\delta_{D(T)})^*)=X^*$
  and use Lemma~\ref{pr:rds}
\end{proof}

Now we will characterize the Fitzpatrick families of bounded range/domain
maximal monotone operators.

\begin{lemma}
  \label{lm:pbr}
  Let $T:X\tos X^*$ be a maximal monotone operator. The following
  conditions are equivalent:
  \begin{enumerate}
  \item $T$ has a bounded range;
  \item for any $h\in\mathcal{F}_T$, $P_2(D(h))$ is bounded; 
  \item there exists $h\in\mathcal{F}_T$ such that $P_2(D(h))$ is bounded;
  \item the family 
    \[
    \{h(\cdot,x^*):X\to\BR,\;x\mapsto h(x,x^*)\;|\;
    h\in\mathcal{F}_T,x^*\in P_2D(h)\}
    \]
    is an equi-Lipschitz family of real valued functions, that is,
    these functions are real-valued and there exists $0\leq L<\infty$ such
    that, 
    \[
    |h(x,x^*)-h(z,x^*)|\leq L\norm{x-z},\qquad
    \forall x,z\in X, h\in\mathcal{F}_T, x^*\in P_2D(h).
    \]
  \item there exists $h\in\mathcal{F}_T$ such that the family
    \[
    \{h(\cdot,x^*):X\to\BR,\;x\mapsto h(x,x^*)\;|\;
    x^*\in P_2D(h)\}
    \]
    is an equi-Lipschitz family of real valued functions, that is,
    these functions are real-valued and there exists $0\leq L<\infty$ such
    that, 
    \[
    |h(x,x^*)-h(z,x^*)|\leq L\norm{x-z},\qquad
    \forall x,z\in X,  x^*\in P_2D(h).
    \]
  \end{enumerate}
\end{lemma}

\begin{proof}
  The equivalence between items 1, 2, and 3 follows trivially from
  Lemma~\ref{lm:dtcl}.
  
  Suppose that item 1 holds, which means that there exists $0\leq L<
  \infty $ such that
  \[
  R(T)\subseteq B_{X^*}[L]
  \]
  Take $h\in \mathcal{F}_T$.  From Lemma~\ref{lm:dtcl} and these inclusion
  it follows that 
  $P_2D(h)\subseteq B_{X^*}[L]$. Using this inclusion and
  Corollary~\ref{cr:bas} we conclude that $P_1D(h^*)\subseteq
  B_{X^*}[L]$. Hence, using also Proposition~\ref{pr:hsb}(item 1) we
  conclude that item 4 holds for such a $L$.

  Item 4 trivially implies item 5. 

  Suppose that item 5 holds for some $h\in\mathcal{F}_T$. Take
  \[
  x^*\in X^*,\qquad \norm{x^*}>L.
  \]
  There exists $(y_0,y_0^*)\in T$ and so, $y_0^*\in P_2D(h)$
  \begin{align*}
    h^*(x^*,x)&=\sup_{(y,y^*)\in X\times X^*}\inner{y}{x^*}+\inner{x}{y^*}-h(y,y^*)
   \\
   &\geq\sup_{y\in X}\;\;\inner{y}{x^*}+\inner{x}{y_0^*}-h(y,y_0^*)
 \\
   &\geq\sup_{y\in X}\;\;\inner{y}{x^*}+\inner{x}{y_0^*}-h(0,y_0^*)-L\norm{y}\\
   &=\inner{x}{y_0^*}-h(0,y_0^*)+\sup_{y\in X}\;\;\inner{y}{x^*}
   -L\norm{y}=\infty
  \end{align*}
  Therefore, $P_2D(\mathcal{J}h)\subseteq B_{X^*}[L]$ and using also
  Theorem~\ref{th:fitz} we conclude that item 1 holds.
\end{proof}

\begin{lemma}
  \label{lm:pbd}
  Let $T:X\tos X^*$ be a maximal monotone operator. The following conditions
  are equivalent:
  \begin{enumerate}
  \item $T$ has a bounded domain;
  \item for any $h\in\mathcal{F}_T$, $P_1(D(h))$ is bounded; 
  \item there exists $h\in\mathcal{F}_T$ such that $P_1(D(h))$ is bounded;
  \item the family 
    \[
    \{h(x,\cdot):X^*\to\BR,\;x^*\mapsto h(x,x^*)\;|\;
    h\in\mathcal{F}_T, h\text{ $s\times w*$ closed},\, x\in P_1D(h)\}
    \]
    is an equi-Lipschitz family of real valued functions.
  \item there exists $h\in\mathcal{F}_T$, $h$ strong$\times$weak$*$ closed, such
    that the family
    \[
    \{h(x,\cdot):X^*\to\BR,\;x^*\mapsto h(x,x^*)\;|\;
    x^*\in P_2D(h)\}
    \]
    is an equi-Lipschitz family of real valued functions.
  \end{enumerate}
\end{lemma}

\begin{proof}
  The equivalence between items 1, 2, and 3 follows trivially from
  Lemma~\ref{lm:dtcl}.
  
  Suppose that item 1) holds. There exists $0\leq L<\infty$ such that
  \[
  D(T)\subseteq B_X[L]
  \]
  Take $h\in \mathcal{F}_T$, $h$ strong$\times$weak-$*$ closed.
  Using Lemma~\ref{lm:dtcl} we conclude that 
  \[
  P_1D(h)\subseteq
  B_{X}[L].
  \]
  Using this inclusion and the second part of Lemma~\ref{lm:dhs} we
  conclude that
  \[
   P_1D(\mathcal{J}h)\subseteq B_X[L].
  \]
  Applying Proposition~\ref{pr:hb}(item 2) to $\mathcal{J}h$ we conclude
  that
  \[
  D((\mathcal{J}h)^*)=X^*\times P_2 D((\mathcal{J}h)^*)
  \]
  and that for any $x^{**}\in P_2D((\mathcal{J}h)^*)$,
  \[
  (\mathcal{J}h)^*(\cdot,x^{**}):X^*\to\R,\;x^*\mapsto
  (\mathcal{J}h)^*(x^*,x^{**})
  \]
  is $L$-Lipschitz continuous. Since $h$ is strong$\times$weak$*$ closed
  \[
  h(x,x^*)=\mathcal{J}^2h(x,x^*)=(\mathcal{J}h)^{*}(x^*,x)
 \]
  Combining the above results, we conclude that item 4 holds.

  If item 4 holds, then item 5 holds for $h=\varphi_T$, which 
  is strong$\times$weak$*$ closed.

   The implication 5$\Rightarrow$1 
  is proved as in Lemma~\ref{lm:pbr}.
\end{proof}

\begin{theorem}
  \label{th:gg}
  If  $T:X\tos X^*$ is maximal monotone, then
  \[
   D(\mathcal{J}\delta_T)\subset \{(x,x^*)\in X\times X^*\;|\;
   \inner{x}{x^*}\leq 0\}.
  \]
\end{theorem}

\begin{proof}
  For any $(w,w^*),(x,x^*)\in X\times X^*$ and $\lambda>0$
  \begin{align*}
    \varphi_T(w,w^*)+\lambda\mathcal{J}\delta_T(x,x^*)
   =\sigma_T^*(w,w^*)+\lambda\delta_{D(\sigma_T)}^*(x,x^*)
  \geq
    \sigma_T^*(w^*+\lambda x^*,x+\lambda w)
  \end{align*}
 where the inequality follows from Proposition~\ref{pr:e}.
 Observe that
 \[
    \sigma_T^*(w^*+\lambda x^*,x+\lambda w)
=\varphi_T(w+\lambda x,    w^*+\lambda x^*)\geq \inner{w+\lambda x}{w^*+\lambda x^*}.
\]
  Therefore
  \[
  \varphi_T(w,w^*)+\lambda\mathcal{J}\delta_T(x,x^*)
  \geq\inner{w+\lambda
    x}{w^*+\lambda x^*}.
  \]
  To end the proof, use  $(w,w^*)\in T$ in the above inequality,
  divide it by $\lambda$ and take the limit
  $\lambda\to\infty$.
\end{proof}

\section{Convex functions which represent bounded-range 
maximal monotone operators
and a non type (D) operator with bounded range}
\label{sec:main}

The next theorem is one of the main results of this work, and
provides sufficient conditions for a convex function to represent
a bounded-range maximal monotone operator.

\begin{theorem}
  \label{th:main}
  If $h:X\times X^*\to\BR$ is convex
  $
  h\geq\pi$, $\mathcal{J}h\geq\pi
  $
  and $P_2(D(h))$ is bounded then
  \begin{enumerate}
  \item $P_2(\mathcal{J}h)=P_1D(h^*)$ is bounded;
  \item $P_1D(h)=X$;
  \item $\mathcal{J}h$ and $\cl_{s\times w*}\;h$ ( the lower
    semi-continuous closure of $h$ in the strong$\times$weak-$*$
    topology) are Fitzpatrick functions of the maximal monotone
    operator
    \[
    T=\{(x,x^*)\;|\;\mathcal{J}h(x,x^*)=\inner{x}{x^*}\};
    \]
  \item  $R(T)$ is bounded and 
    $D(T)=X$.
\end{enumerate}
\end{theorem}

\begin{proof}
  Since the assumptions of the theorem remain valid if $h$ is replaced
  by $\cl\,h$ and $(\cl\,h)^*=h^*$, we may assume that $h$ is proper
  closed convex function.

  Item 1 follows from the assumption $\mathcal{J}h\geq\pi$ and
  Corollary~\ref{cr:bas}. Item 2 follows from item 1,the assumption
  $h\geq\pi$ and Corollary~\ref{cr:bas2}.

  Item 3 follows item 2, the assumption $h\geq\pi$,
  $\mathcal{J}h\geq\pi$ and Theorem~\ref{th:hhs}.

  The first statement on item 4 follows from items 3 and 1.
  The second part follows from the first, the maximal monotonicity of
  $T$, and Debrunner-Flor Theorem~\cite{MR0170189}.
\end{proof}

Zagrodny~\cite{MR2595075} proved that maximal monotone operators with
a relatively compact range (in the strong topology) are of type (D).
This result cannot be extended to the weak-$*$ topology, that is, to
maximal monotone operators with a relatively weak-$*$ compact range.
In this section we provide an example of a bounded range, non type (D)
maximal monotone operator.

\begin{theorem}
  \label{th:cw}
  Let $X$ be a non-reflexive real Banach space.  Endow $Z=X\times X^*$
  with the norm
  \[
  \norm{(x,x^*)}=\sqrt{\norm{x}^2+\norm{x^*}^2}
  \]
  and define
  \begin{equation}
    \label{eq:h}
    h:Z\times Z^*\to\BR,\;\;
    h((x,x^*),(y^*,y^{**}))=
    \norm{(x-y^{**},x^*+y^*)}+\delta_{B_{Z^*}[1]}(y^*,y^{**})
    +\delta_{X}(y^{**}).
  \end{equation}
  Then:
  \begin{enumerate}
  \item $h$ is a closed convex function, $P_2D(h)$ is bounded,
    $h\geq\pi$, $\mathcal{J}h\geq\pi$;
  \item $ T=\{(z,z^*)\in Z\times
    Z^*\;|\;\mathcal{J}h(z,z^*)=\inner{z}{z^*}\}$ is a bounded-range
    maximal monotone operator, $\mathcal{J}h\in\mathcal{F}_T$;
  \item $T$ is not of type (D);
  \end{enumerate}
\end{theorem}

\begin{proof}
  Note that $Z^*=X^*\times X^{**}$, $Z^{**}=X^{**}\times X^{***}$ and
  their corresponding norms are, respectively
  \[
  \norm{(x^*,x^{**})}=\sqrt{\norm{x^*}^2+\norm{x^{**}}^2},\quad
  \norm{(x^{**},x^{***})}=\sqrt{\norm{x^{**}}^2+\norm{x^{***}}^2}.
  \]
  The function $h$ is trivially proper, convex, lower semi-continuous
  and
  \begin{align}
    \nonumber
    D(h)&=\{ ((x,x^*),(y^*,y^{**}))\;|\; (x,x^*)\in X\times X^*,
    \; (y^*,y^{**})\in X^*\times X,\;\norm{(y^*,y^{**})}\leq 1\}\\
    &=(X\times X^*)\times ((X^*\times X)\cap B_{X^*\times X^{**}} [1]).
    \label{eq:dh}
  \end{align}
  Hence $P_2D(h)$ is bounded.
  The conjugate of $h$ is the function $h^*:Z^*\times Z^{**}\to\BR$,
  \begin{align*}
    h^*((p^*,p^{**}),(q^{**},q^{***}))=
    \sup_{\stackrel{x,y\in X, x^*,y^*\in X^*,}{\norm{(y^*,y)}\leq 1}}&
    \inner{x}{p^*}+\inner{p^{**}}{x^*}+\inner{q^{**}}{y^*}
    +\inner{y}{q^{***}}-\norm{(x-y,x^*+y^*)}\\
    =
    \sup_{\stackrel{x,y\in X, x^*,y^*\in X^*,}{ \norm{(y^*,y)}\leq 1}}&
    \inner{x-y}{p^*}+\inner{p^{**}}{x^*+y^*}-\norm{(x-y,x^*+y^*)}\\
    &+\inner{y}{q^{***}+p^*}+\inner{q^{**}-p^{**}}{y^*}.
  \end{align*}
  Whence
  \begin{align}
    \nonumber
    h^*((p^*,p^{**}),(q^{**},q^{***}))&=
    \sup_{\stackrel{y\in X,y^*\in X^*,}{\norm{(y^*,y)}}\leq 1}
    \delta_{B_{Z^*}}(p^*,p^{**})
    +\inner{y}{q^{***}+p^*}+\inner{q^{**}-p^{**}}{y^*}\\
    \label{eq:hs}
    &=  \delta_{B_Z^*}(p^*,p^{**})+\norm{(q^{**}-p^{**},q^{***}|_X+p^*)}
  \end{align}
  where $q^{***}|_X$ stands for the functional $q^{***}$ restricted to $X$.

   Now we will prove that $h,\mathcal{J}h\geq \pi$.
  With this aim, take
  \[
 (z,z^*)=((x,x^*),(y^*,y^{**}))\in Z\times Z^*.
  \]
  First note that
  \begin{align*}
    \inner{(x,x^*)}{(y^*,y^{**})}=\inner{x}{y^*}+\inner{y^{**}}{x^*}
&=\inner{x-y^{**}}{y^*}+\inner{y^{**}}{x^*+y^*}\\
    &\leq\norm{x-y^{**}}\norm{y^*}+\norm{x^*+y^*}\norm{y^{**}}\\
    &\leq \sqrt{\norm{x-y^{**}}^2+\norm{x^*+y^*}^2}\;
    \sqrt{\norm{y^*}^2+\norm{y^{**}}^2}\\
 &=\norm{(x-y^{**},x^*+y^*)}\;\norm{(y^*,y^{**})}
  \end{align*}
  where the second inequality follows from Cauchy-Schwarz inequality
  in $\R^2$.
  Using \eqref{eq:hs} we have
  \begin{align*}
    \mathcal{J}h(z^*,z)=h^*((y^*,y^{**}),(x,x^*))&=\delta_{B_{Z^*}[1]}(y^*,y)
    +\norm{(x-y^{**},x^*+y^*)},
  \end{align*}
  Therefore $\mathcal{J}h\geq\pi$.  Direct comparison of the above
  equation with \eqref{eq:h} shows that
  $h\geq\mathcal{J}h$. Therefore, we also have $h\geq \pi$, which
  completes the proof of item 1.

  Item 2 follows from item 1 and Theorem~\ref{th:main}.

  Now we will prove that $h^*$ does not majorizes the duality product
  in $Z^*\times Z^{**}$. Since $X$ is non-reflexive, there exists
  \[
  p_0^{**}\in X^{**}\setminus X,\qquad \norm{p_0^{***}}=1.
  \]
  Direct application of Hahn-Banach Lemma shows that there exists
  $q_0^{***}\in X^{***}$ such that
  \[
  q_0^{***}|X\equiv 0,\qquad \inner{q_0^{**}}{q_0^{***}}=1.
  \] 
  Observe that
  \[
1=  \inner{(0,q_0^{**})}{(q_0^{**},q_0^{***})}>
h^*((0,q_0^{**}),(q_0^{**},q_0^{***}))=0.  \]
  Combining these results with the inclusion $h\in\mathcal{F}_T$ and
 \cite[Theorem 4.4]{MR2731293} we conclude that $T$ is not of type (D).
\end{proof}

Observe that $h$ defined in Theorem~\ref{th:cw} is given by
\[
h(x,x^*)=\inf_{u\in X} h_1(x-u,x^*)+h_2(u,x^*)
\]
with $h_1$ Fencehel-Young function for $\partial \norm{\cdot}$,
\[
h_1(z,z^*)=\norm{z}+\delta_{B_Z[1]}(x^*),
\]
and $h_2=\delta_{T_2}$, where $T_2$ is the non-type (D) linear isometry
studied in~\cite{nodliniso},
\[
T_2:X\times X^*\to X\times X^*,\quad T(x,x^*)=(-x^*,x).
\]

\appendix
\section{Appendix}
\begin{proof}[Proof of Proposition \ref{pr:e}]
  For any $y\in X$
  \begin{align*}
    \delta_{D(f)}(w^*)\geq \inner{y}{w^*}-\delta_{D(f)}(y),\qquad
    f^*(z^*)\geq \inner{y}{z^*}-f(z).
  \end{align*}
  Multiplying the first inequality by $\lambda\geq 0$ and adding it to
  the second inequality we obtain
  \begin{align*}
    f(z^*)+\lambda\delta_{D(f)}^*(w^*)\geq \inner{y}{z^*+\lambda w^*}-f(y),
  \end{align*}
  and the conclusion follows taking the $\sup$ in $y$ at the right
  hand side of the above inequality
\end{proof}

\end{document}